\documentclass[11pt]{amsart}
\usepackage{amsmath}
\usepackage{amssymb}
\usepackage{amsfonts}
\usepackage{amsthm}
\usepackage{enumerate}
\usepackage[mathscr]{eucal}
\usepackage{verbatim}
\usepackage{amsthm}
\usepackage{amscd}

\newtheorem{theorem}{Theorem}[section]

\newtheorem{lemma}[theorem]{Lemma}

\theoremstyle{definition}

\newtheorem{remark}{Remark}

\newtheorem{innercustomgeneric}{\customgenericname}
\providecommand{\customgenericname}{}

\newcommand{\newcustomtheorem}[2]{%
  \newenvironment{#1}[1]
  {%
   \renewcommand\customgenericname{#2}%
   \renewcommand\theinnercustomgeneric{##1}%
   \innercustomgeneric
  }
  {\endinnercustomgeneric}
}

\newcustomtheorem{customthm}{Theorem}

\newcommand\relphantom[1]{\mathrel{\phantom{#1}}}

\numberwithin{equation}{section}

\begin{document}

\address{School of Mathematics \\
           Korea Institute for Advanced Study, Seoul\\
           Republic of Korea}
   \email{qkrqowns@kias.re.kr}

\author{Bae Jun Park}

\title[Function spaces]{Boundedness of pseudo-differential operators of type $(0,0)$ on Triebel-Lizorkin and Besov spaces}

\begin{abstract} 
In this work we establish sharp boundedness results for pseudo-differential operators corresponding to $a\in\mathcal{S}_{0,0}^{m}$ on Triebel-Lizorkin spaces $F_p^{s,q}$ and Besov spaces $B_p^{s,q}$. 
\end{abstract}

\maketitle

\section{\textbf{Introduction}}\label{intro}

Let $S(\mathbb{R}^d)$ denote the Schwartz space and $S'(\mathbb{R}^d)$ the space of tempered distributions. For the Fourier transform of $f\in S(\mathbb{R}^d)$ we use the definition $\widehat{f}(\xi):=\int_{\mathbb{R}^d}{f(x)e^{-2\pi i\langle x,\xi\rangle}}dx$ and denote by $f^{\vee}$ the inverse Fourier transform of $f$. We also extend these transforms to the space of tempered distributions.

For $0\leq \delta\leq \rho\leq 1$ and $m\in\mathbb{R}$  a symbol $a$ in H\"ormander's class $\mathcal{S}_{\rho,\delta}^{m}$ is a smooth function defined on $\mathbb{R}^d\times \mathbb{R}^d$ satisfying that for all multi-indices $\alpha$ and $\beta$ there exists a constant $C_{\alpha,\beta}$ such that 
\begin{equation*}
\big| \partial_{\xi}^{\alpha}\partial_{x}^{\beta}a(x,\xi)\big|\leq C_{\alpha,\beta}\big( 1+|\xi|\big)^{m-\rho|\alpha|+\delta |\beta|},\quad x,\xi\in\mathbb{R}^d,
\end{equation*}
and the corresponding pseudo-differential operator $T_{[a]}$ is given by 
\begin{equation*}
T_{[a]}f(x):=\int_{\mathbb{R}^d}{a(x,\xi)\widehat{f}(\xi)e^{2\pi i\langle x,\xi\rangle}}d\xi,\quad f\in S(\mathbb{R}^d).
\end{equation*}

The operator $T_{[a]}$ is well-defined on $S(\mathbb{R}^d)$ and  it maps $S(\mathbb{R}^d)$ continuously into itself. For $(\rho,\delta)\not= (1,1)$  the operators form a class invariant under taking adjoints, and thus we have $T_{[a]}:S'(\mathbb{R}^d)\to S'(\mathbb{R}^d)$ via duality. See \cite{Ho,Ho1,Park} for details.

For $0\leq\delta\leq \rho<1$ the boundedness of $T_{[a]}\in Op\mathcal{S}_{\rho,\delta}^{m}$ was studied, for example, by Calder\'on and Vaillancourt in \cite{Ca_Va}, by Fefferman in \cite{Fe}, and by P\"aiv\"arinta and Somersalo in \cite{Pa_So}.
The operators are bounded on $h^p$ for $0<p<\infty$ if
\begin{equation*}
m\leq -d(1-\rho)\big| 1/2-1/p\big|.
\end{equation*}
The author \cite{Park} generalized this result to Triebel-Lizorkin and Besov spaces for $0<\rho<1$.
\begin{customthm}{A}\label{theorema}
Let $0 <  \rho <1$, $0<p,q,t\leq \infty$, and $s_1,s_2 \in \mathbb{R}$. Suppose $m\in\mathbb{R}$ satisfies \begin{equation}\label{ccexample}
m-s_1+s_2\leq -d(1-\rho)\big|  1/2-1/p \big|
\end{equation} and $a\in\mathcal{S}_{\rho,\rho}^m$. Then $T_{[a]}$ maps $F_p^{s_1,q}(\mathbb{R}^d)$ into $F_p^{s_2,t}(\mathbb{R}^d)$ if one of the following cases holds;
\begin{enumerate}
\item if  $m-s_1+s_2<-d(1-\rho)\big| 1/2-1/p  \big|$,

\item if  $p=2$, $q\leq 2\leq t$, and $m-s_1+s_2=0$,

\item if $0<p<2$, $p\leq t\leq\infty$, $0<q\leq \infty$, and $m-s_1+s_2=-d(1-\rho)\big(1/p-1/2\big)$,

\item if $2< p\leq \infty$, $0<t\leq\infty$, $0<q\leq p$,  and $m-s_1+s_2=-d(1-\rho)\big(1/2-1/p\big)$.

\end{enumerate}
\end{customthm}

\begin{customthm}{B}\label{theoremb}
Let $0 <  \rho <1$, $0<p,q,t\leq \infty$, and $s_1,s_2 \in \mathbb{R}$. Suppose $m\in\mathbb{R}$ satisfies (\ref{ccexample}) and $a\in\mathcal{S}_{\rho,\rho}^m$. Then $T_{[a]}$ maps $B_p^{s_1,q}(\mathbb{R}^d)$ into $B_p^{s_2,t}(\mathbb{R}^d)$ if one of the following cases holds;
\begin{enumerate}
\item if  $m-s_1+s_2<-d(1-\rho)\big| 1/2-1/p  \big|$,

\item if  $q\leq t$, and $m-s_1+s_2=-d\big|1/2-1/p \big|$.

\end{enumerate}
\end{customthm}

The key idea to prove Theorem \ref{theorema} (for $2<p\leq \infty$) is an estimate for a family of operators that is reminiscent of Calder\'on-Zygmund theory in \cite{Pr_Ro_Se}, but the argument does not work when $\rho=0$.
In order to illustrate this, take a simple case in which $a(D)$ is a multiplier operator
 corresponding to $a(x,\xi)=a(\xi)$ in $\mathcal{S}_{\rho,0}^m$.
 By applying (inhomogeneous) Littlewood-Paley operators $\{\Lambda_k\}_{k=0}^{\infty}$, defined in Section \ref{functionspace},
we decompose $a(D)f=\sum_{k=0}^{\infty}{\Lambda_ka(D)f}$ and observe that the kernel $K_k(x,y)$ of the operator $\Lambda_ka(D)$ satisfies the size estimate
\begin{equation}\label{keyideaprevious}
|K_k(x,y)|\lesssim_M 2^{-k(\rho M-m)}\frac{1}{|x-y|^M}, \qquad |x-y|\gtrsim 1
\end{equation} by using integration by parts $M$ times. If $\rho>0$, by taking $M$ sufficiently large , one has very nice decay $2^{-k(\rho M-m)}$ (even though $\rho$ is very small). Actually, the estimate (\ref{keyideaprevious}) is an essential part in the proof of Theorem \ref{theorema}. However, when $\rho=0$ the decay property is no longer available and an alternative argument will be needed in this case.

Moreover, Theorem \ref{theorema} and \ref{theoremb} are sharp in the sense that the condition (\ref{ccexample}) is  necessary and when the equality of (\ref{ccexample}) holds the assumptions on $q,t$ are necessary. 
To be specific, the boundedness results fail with the oscillatory multiplier operator 
\begin{equation}\label{multiplierexample}
c_{m,\rho}(D)=\dfrac{e^{ -2\pi i |D|^{{(1-\rho)}}}}{(1+|D|^2)^{-{m}/{2}}}
\end{equation} if the assumptions do not work.
However, when $\rho=0$ this does not hold anymore. Indeed, it is known in \cite[Theorem 4.2]{Mi} that $c_{m,0}(D)$ is bounded on $h^p(\mathbb{R}^d)\big(=F_p^{0,2}(\mathbb{R}^d)\big)$ if and only if $m\leq-(d-1)\big|1/2-1/p \big|$ and therefore the operator does not provide a sharp boundedness estimate for  $\rho=0$.

In this paper we extend Theorem \ref{theorema} and \ref{theoremb} to $\rho=0$ and construct new counterexamples to achieve the sharpness of our results.
The main results of this paper are the following theorems.
\begin{theorem}\label{main}
Let $0<p,q,t\leq\infty$ and $s_1,s_2\in\mathbb{R}$. Suppose $m\in\mathbb{R}$ satisfies
\begin{equation}\label{ccexample2}
m-s_1+s_2\leq -d\big|1/2-1/p \big|
\end{equation} and $a\in\mathcal{S}_{0,0}^{m}$. Then $T_{[a]}$ maps $F_p^{s_1,q}(\mathbb{R}^d)$ into $F_p^{s_2,t}(\mathbb{R}^d)$ if one of the following cases holds;
\begin{enumerate}
\item if  $m-s_1+s_2<-d\big| 1/2-1/p  \big|$,

\item if  $p=2$, $q\leq 2\leq t$, and $m-s_1+s_2=0$,

\item if $0<p<2$, $p\leq t\leq\infty$, $0<q\leq \infty$, and $m-s_1+s_2=-d\big(1/p-1/2\big)$,

\item if $2< p\leq \infty$, $0<t\leq\infty$, $0<q\leq p$,  and $m-s_1+s_2=-d\big(1/2-1/p\big)$.
\end{enumerate}

\end{theorem}

\begin{theorem}\label{main2}
Let $0<p,q,t\leq \infty$, and $s_1,s_2 \in \mathbb{R}$. Suppose $m\in\mathbb{R}$ satisfies (\ref{ccexample2}) and $a\in\mathcal{S}_{0,0}^m$. Then $T_{[a]}$ maps $B_p^{s_1,q}(\mathbb{R}^d)$ into $B_p^{s_2,t}(\mathbb{R}^d)$ if one of the following cases holds;
\begin{enumerate}
\item if  $m-s_1+s_2<-d(1-\rho)\big| 1/2-1/p  \big|$,

\item if  $q\leq t$, and $m-s_1+s_2=-d\big|1/2-1/p \big|$.

\end{enumerate}
\end{theorem}

Theorem \ref{main} and \ref{main2} are sharp in the following sense (except the case $p=\infty$ in Theorem \ref{main2} (2)).

\begin{theorem}\label{sharptheorem}
Let $0<p,q,t\leq \infty$, $s_1,s_2\in\mathbb{R}$ and $m\in\mathbb{R}$. Then there exists a symbol $a\in\mathcal{S}_{0,0}^{m}$ such that
\begin{equation*}
\Vert T_{[a]}\Vert_{F_p^{s_1,q}(\mathbb{R}^d)\to F_p^{s_2,t}(\mathbb{R}^d)}=\infty
\end{equation*} if one of the following conditions holds;
\begin{enumerate}
\item if  $m-s_1+s_2>-d\big| 1/2-1/p  \big|$,

\item if $m-s_1+s_2=-d\big(1/p-1/2\big)$, $0<p\leq 2$, $0<q\leq \infty$, and $t<p$,

\item if $m-s_1+s_2=-d\big(1/2-1/p\big)$, $2\leq p<\infty$, $0<t\leq \infty$, and $p<q$.
\end{enumerate}
\end{theorem}

\begin{theorem}\label{sharptheorem2}
Let $0<p,q,t\leq \infty$, $s_1,s_2\in\mathbb{R}$ and $m\in\mathbb{R}$. Then there exists a symbol $a\in\mathcal{S}_{0,0}^{m}$ such that
\begin{equation*}
\Vert T_{[a]}\Vert_{B_p^{s_1,q}(\mathbb{R}^d)\to B_p^{s_2,t}(\mathbb{R}^d)}=\infty
\end{equation*} if one of the following conditions holds;
\begin{enumerate}
\item if  $m-s_1+s_2>-d\big| 1/2-1/p  \big|$,

\item if $q>t$, $p<\infty$ and $m-s_1+s_2=-d\big|1/2-1/p\big|$.

\end{enumerate}
\end{theorem}

\begin{remark}
To complete the sharpness of Theorem \ref{main2} it remains to show the necessity of $q\leq t$ when $p=\infty$ and $m-s_1+s_2=-d/2$, but we couldn't resolve this issue.
\end{remark}

We point out that the cases $m-s_1+s_2<-d\big|1/2-1/p \big|$ in Theorem \ref{main} and \ref{main2} follow from $h^p$ boundedness of $T_{[a]}$ in \cite{Pa_So} and the use of proper embeddings, and so does Theorem \ref{main} (2).  Therefore, we will be concerned only with $(3)$ and $(4)$ in Theorem \ref{main} and $(2)$ in Theorem \ref{main2}. Furthermore, it suffices to prove that 
\begin{align}
\big\Vert T_{[a]}f\big\Vert_{F_p^{s_2,p}(\mathbb{R}^d)}&\lesssim \Vert f\Vert_{F_p^{s_1,\infty}(\mathbb{R}^d)} \quad \text{ for }~ 0<p\leq 1 \tag{Theorem \ref{main} (3)*} \label{bigclaim1}, \\
\big\Vert T_{[a]}f\big\Vert_{F_p^{s_2,t}(\mathbb{R}^d)}&\lesssim \Vert f\Vert_{F_p^{s_1,p}(\mathbb{R}^d)}  \quad \text{ for }~ 2<p\leq \infty,~ 0<t<1, \tag{Theorem \ref{main} (4)*}\label{bigclaim2}\\
\big\Vert T_{[a]}f\big\Vert_{B_p^{s_2,q}(\mathbb{R}^d)}&\lesssim \Vert f\Vert_{B_p^{s_1,q}(\mathbb{R}^d)}   \quad \text{ for }~ 0<q\leq \infty \tag{Theorem \ref{main2} (2)*}\label{bigclaim3}
\end{align} 
due to the embedding $F_p^{s,q}\hookrightarrow F_p^{s,t}$ for $q\leq t\leq \infty$. We will not pursue the case $1<p<2$ here because it clearly follows from the duality consideration in \cite[p556]{Park}.

 (\ref{bigclaim1}) can be proved by the method similar to that used in the proof of Theorem \ref{theorema} where we applied discrete characterization of $F_p^{s,q}$ via the Frazier-Jawerth's $\varphi$-transform and atomic decomposition for the discrete spaces.  
We are mainly interested in (\ref{bigclaim2}) (as the previous argument, like the decay property in (\ref{keyideaprevious}), is not applicable here). The proof is based on the characterization of a vector-valued function space via a certain sharp maximal function, stated in Lemma \ref{54}, and 
the $L^{\infty}$-$L_{avg}^2$ estimate in Lemma \ref{inftylemma}, which is key in the proof.

For the negative results, assume $m-s_1+s_2>-d\big|1/2-1/p\big|$ and choose $0<\rho<1$ such that $m-s_1+s_2>-d(1-\rho)\big|1/2-1/p \big|$. Then the oscillatory multiplier $c_{m,\rho}(\in\mathcal{S}_{\rho,0}^m\subset \mathcal{S}_{0,0}^{m})$ in (\ref{multiplierexample}) proves Theorem \ref{sharptheorem} (1)
(or Theorem \ref{sharptheorem} (1) follows immediately from Theorem \ref{sharptheorem}  (2) and (3) for $0<p<\infty$).
Theorem \ref{sharptheorem2} (1) also follows in a similar argument.
Therefore only  the endpoint case $m-s_1+s_2=-d\big|1/2-1/p\big|$ will be considered. We will apply the randomization technique of Christ and Seeger \cite{Ch_Se} to construct counterexamples for Theorem \ref{sharptheorem} (2) and the duality argument for Theorem \ref{sharptheorem} (3). The proof of Theorem \ref{sharptheorem2} (2) is ``relatively" easier, just involving Khintchine's inequality.

This paper is organized in the following way.
Section \ref{functionspace} is dedicated to preliminaries, introducing definitions and general properties about Besov and Triebel-Lizorkin spaces. 
In Section \ref{proofmain} and \ref{proofmain2} we provide the proof of Theorem \ref{main} and \ref{main2}. In the last section we construct some examples to prove Theorem \ref{sharptheorem} and \ref{sharptheorem2}.

We make some convention on notation. Let $\mathbb{N}$ and $\mathbb{Z}$ be the collections of all natural numbers and all integers, respectively, and $\mathbb{N}_0:=\mathbb{N}\cup \{0\}$.
For the sake of simplicity we restrict ourselves in the sequel to function spaces defined on $\mathbb{R}^d$ and omit ``$\mathbb{R}^d$''. In other words, $S$, $S'$, $B_p^{s,q}$, and $F_p^{s,q} $ stand for $S(\mathbb{R}^d)$, $S'(\mathbb{R}^d)$, $B_p^{s,q}(\mathbb{R}^d)$, and $F_p^{s,q}(\mathbb{R}^d)$, respectively. Let  $\mathcal{D}$ stand for the set of all dyadic cubes in $\mathbb{R}^d$ and $\mathcal{D}_{k}$ the subset of $\mathcal{D}$ consisting of the cubes with side length $2^{-k}$ for $k\in\mathbb{Z}$. 
  For $Q\in \mathcal{D}$, denote the side length of $Q$ by $l(Q)$ and the characteristic function of $Q$ by $\chi_Q$.

\section{Preliminaries}\label{functionspace}

\subsection{Function spaces}
Let $\Phi\in S$ satisfy $Supp(\widehat{\Phi})\subset \big\{\xi\in\mathbb{R}^d:|\xi|\leq 2 \big\}$ and $\widehat{\Phi}(\xi)=1$ for $|\xi|\leq 1$. Let $\phi:=\Phi-2^{-d}\Phi(2^{-1}\cdot)$, and define $\phi_0:=\Phi$ and $\phi_k:=2^{kd}\phi(2^k\cdot)$ for $k\geq 1$.
Then $ \{\phi_k\}_{k\in\mathbb{N}_0}$  forms inhomogeneous Littlewood-Paley partition of unity.
Note that $Supp(\widehat{\phi_k})\subset \big\{\xi\in\mathbb{R}^d: 2^{k-1}\leq |\xi|\leq 2^{k+1}\big\}$ for $k\geq 1$  and $\sum_{k\in\mathbb{N}_0}{\widehat{\phi_k}}=1$.
We define a convolution operator   $\Lambda_kf:=\phi_k\ast f$ for $k\in \mathbb{N}_0$.  
Then for $s\in \mathbb{R}$ and $0<p,q\leq \infty$, $B_p^{s,q}$ and $F_p^{s,q}$ are the collection of all $f\in S'$ such that
\begin{equation*}
\Vert f\Vert_{B_p^{s,q}}:=\big\Vert \big\{ 2^{sk}\Lambda_kf\big\}_{k=0}^{\infty}\big\Vert_{l^q(L^p)}<\infty,
\end{equation*} 
\begin{equation*}
\Vert f\Vert_{F_p^{s,q}}:=\big\Vert \big\{ 2^{sk}\Lambda_kf\big\}_{k=0}^{\infty}\big\Vert_{L^p(l^q)}<\infty, \quad p<\infty,
\end{equation*} respectively.
When $p=q=\infty$ we employ $F_{\infty}^{s,\infty}=B_{\infty}^{s,\infty}$, and
$F_{\infty}^{s,q}$, $0<q<\infty$, is the collection of tempered distributions $f$ with
\begin{equation*}
\Vert f\Vert_{F_{\infty}^{s,q}}:=\Vert \Lambda_0f\Vert_{L^{\infty}}+\sup_{P\in\mathcal{D},l(P)<1}\Big(\frac{1}{|P|}\int_P{\sum_{k=-\log_2{l(P)}}^{\infty}{2^{s kq}|\Lambda_kf(x)|^q}}dx \Big)^{1/q}<\infty
\end{equation*}
where the supremum is taken over all dyadic cubes whose side length is less than $1$.

Then these spaces provide a general framework that unifies classical function spaces.
\begin{align*}
& L^p\text{space} &{F}_p^{0,2}=L^p & &1<p<\infty \\
&\text{local Hardy space} & F_p^{0,2}=h^p & &0<p<\infty\\
&\text{Sobolev space}&  {F}_p^{{s},2}=L^p_{{s}}  & & {s} >0, 1<p<\infty\\
& bmo &  \quad  {{F}}_{{\infty}}^{0,2}=bmo.
\end{align*}
Note that $L^p=h^p$ for $1<p<\infty$.

\subsection{Maximal inequalities}
Denote by $\mathcal{M}$ the Hardy-Littlewood maximal operator and let $\mathcal{M}_tu=\big(\mathcal{M}(|u|^t)\big)^{1/t}$ for $0<t<\infty$. For $r>0$ let $\mathcal{E}(r)$ denote the space of all distributions whose Fourier transforms are supported in $\{\xi:|\xi|\leq 2r\}$.
A crucial tool in theory of function spaces is a maximal operator introduced by Peetre \cite{Pe}.
For $r>0$ and $\sigma>0$ define 
\begin{equation*}
\mathfrak{M}_{\sigma,r}u(x)=\sup_{y\in\mathbb{R}^d}{\dfrac{|u(x+y)|}{(1+r|y|)^{\sigma}}}.
\end{equation*} 
As shown in \cite{Pe}, one has the majorization \begin{equation*}
\mathfrak{M}_{\sigma,r}u(x)\lesssim \mathcal{M}_tu(x)
\end{equation*} for all $\sigma\geq d/{t}$ if $u\in \mathcal{E}(r)$.
These estimates imply the following maximal inequality via the Fefferman-Stein's vector-valued inequality in \cite{Fe_St}.
Suppose $0<p<\infty$ and $0<q\leq \infty$. Then for $k\in\mathbb{Z}$ and $A>0$ one has
\begin{equation}\label{max}
\Big\Vert  \Big(\sum_{k}{(\mathfrak{M}_{\sigma,2^k}u_k)^q}\Big)^{1/{q}} \Big\Vert_{L^p} \lesssim_A  \Big\Vert \Big( \sum_{k}{|u_k|^q}  \Big)^{1/{q}}  \Big\Vert_{L^p} ~\text{for}~ \sigma>\max{\big\{d/p,d/q\big\}}
\end{equation} if $u_k\in \mathcal{E}(A2^k)$.
Moreover,  it is proved in \cite{Park2} that
for $\mu\in\mathbb{Z}$, $P\in\mathcal{D}_{\mu}$, and $A>0$ one has
\begin{equation}\label{inftymaximal}
\Big( \frac{1}{|P|}\int_P{\sum_{k=\mu}^{\infty}{\big(\mathfrak{M}_{\sigma,2^k}u_k(x)\big)^q}}dx\Big)^{1/q}\lesssim_A\sup_{R\in\mathcal{D}_{\mu}}{\Big(\frac{1}{|R|}\int_R{\sum_{k=\mu}^{\infty}{|u_k(x)|^q}}dx\Big)^{1/q}} ~\text{for}~ \sigma>d/q
\end{equation} where the constant in the inequality is independent of $\mu$ and $P$.
The condition $\sigma>\max{(d/p,d/q)}$ in (\ref{max}) and (\ref{inftymaximal}) is necessary for the inequalities to hold. We refer the reader to \cite{Ch_Se, Park2} for more details.

\subsection{$\varphi$-transform of $F$-spaces}\label{decomposition}
Suppose $0<p<\infty$, $0<q\leq\infty$, and $s\in\mathbb{R}$. 
For a sequence of complex numbers $b=\{b_Q\}_{Q\in\mathcal{D}, l(Q)\leq 1}$ we define 
\begin{equation*}
\Vert b \Vert_{f_p^{s,q}}:=\big\Vert  g^{s,q}(b)  \big\Vert_{L^p}, \quad 0<p<\infty ~\text{ or }~ p=q=\infty
\end{equation*}
where
\begin{equation*}
g^{s,q}(b)(x):=\Big(\sum_{Q\in\mathcal{D}, l(Q)\leq 1}{\big(|Q|^{-s/{d}-1/2}|b_Q|\chi_Q(x)\big)^q}\Big)^{1/q}.
\end{equation*} 
Then Triebel-Lizorkin space $F_p^{s,q}$ can be characterized by discrete function space $f_p^{s,q}$.
For $c>0$ let  $\vartheta_0,\vartheta, \widetilde{\vartheta}_0, \widetilde{\vartheta} \in\mathcal{S}$ satisfy 
\begin{equation*}
Supp(\widehat{\vartheta}_0), Supp(\widehat{\widetilde{\vartheta}_0})\subset \{\xi : |\xi|\leq 2\}, 
\end{equation*}
\begin{equation*}
Supp(\widehat{\vartheta}), Supp(\widehat{\widetilde{\vartheta}})\subset \{\xi : 1/{2}\leq |\xi|\leq 2\}
\end{equation*}
\begin{equation*}
|\widehat{\vartheta_0}(\xi)|, |\widehat{\widetilde{\vartheta}_0}(\xi)| \geq c>0 ~\text{for}~ |\xi|\leq 5/{3}
\end{equation*}
\begin{equation*}
|\widehat{\vartheta}(\xi)|, |\widehat{\widetilde{\vartheta}}(\xi)| \geq c>0 ~\text{for}~ 3/4\leq |\xi|\leq 5/3
\end{equation*}
 and 
 \begin{equation*}
 \sum_{k=0}^{\infty}{\overline{\widehat{\widetilde{\vartheta}_k}(\xi)}\widehat{\vartheta_k}(\xi)}=1
 \end{equation*} where $\vartheta_k(x)=2^{kd}\vartheta(2^kx)$  and $\widetilde{\vartheta}_k(x)=2^{kd}\widetilde{\vartheta}(2^kx)$ for $k\geq 1$.
For each $Q\in\mathcal{D}$ let $x_Q$ be the lower left corner of $Q$.
 Every $f\in F_p^{s,q}$ can be decomposed as 
\begin{equation}\label{decomposition1}
f(x)=\sum_{Q\in\mathcal{D}, l(Q)\leq 1}{v_Q\vartheta^Q(x)} 
\end{equation} where $\vartheta^Q(x):=|Q|^{1/2}\vartheta_k(x-x_Q)$, $\widetilde{\vartheta}^Q(x):=|Q|^{1/2}\widetilde{\vartheta}_k(x-x_Q)$ for $Q\in\mathcal{D}_k$, and $v_Q:=\langle f,\widetilde{\vartheta}^Q\rangle$.
Moreover, in this case one has \begin{equation*}
\Vert v   \Vert_{f_p^{s,q}} \lesssim \Vert  f  \Vert_{F_p^{s,q}}.
\end{equation*}
The converse estimate also holds.  For any sequence $v=\{v_Q\}_{Q\in\mathcal{D}}$ of complex numbers satisfying $\Vert  v \Vert_{f_p^{s,q}}<\infty$,  
\begin{equation*}
f(x):=\sum_{Q\in\mathcal{D}, l(Q)\leq 1}{v_Q\vartheta^Q(x)}
\end{equation*} belongs to $F_p^{s,q}$ and 
\begin{equation*}
\Vert  f  \Vert_{F_p^{s,q}} \lesssim \Vert  v \Vert_{f_p^{s,q}}.
\end{equation*}
See \cite{Fr_Ja, Fr_Ja1, Fr_Ja2} for more details.

\subsection{Atomic decomposition of ${f}_p^{s,q}$ by nonsmooth $\infty$-atoms }

Let $0<p\leq 1$, $0< q\leq \infty$, and $s\in\mathbb{R}$. A sequence of complex numbers $r=\{r_Q\}_{\substack{Q\in\mathcal{D}\\l(Q)\leq 1}}$ is called a nonsmooth $\infty$-atom for $f_p^{s,q}$ if there exists a dyadic cube $Q_0$ such that 
\begin{equation*}
r_Q=0 \quad \text{if}\quad Q \not\subset Q_0
\end{equation*}
 and \begin{equation*}
\big\Vert  g^{s,q}(r)  \big\Vert_{L^{\infty}}\leq |Q_0|^{-{1}/{p}}.
\end{equation*}
Then we will use the following atomic decomposition of $f_p^{s,q}$ as a substitute of the atomic decomposition of $h^p$ for $0<p\leq 1$.
\begin{lemma}\label{decomhardy}
Suppose $0<p\leq 1$, $p\leq q\leq\infty$, and $b=\{b_Q\}_{Q\in\mathcal{D},l(Q)\leq1}\in f_p^{s,q}$. Then there exist $C_{d,p,q}>0$, a sequence of scalars $\{\lambda_j\}$, and a sequence of $\infty$-atoms $r_j=\{r_{j,Q}\}_{Q\in\mathcal{D}, l(Q)\leq 1}$ for $f_p^{s,q}$ such that 
\begin{equation*}
b=\{b_Q\}=\sum_{j=1}^{\infty}{\lambda_j\{r_{j,Q}\}}=\sum_{j=1}^{\infty}{\lambda_j r_j}
\end{equation*} and  
\begin{equation*}
\Big(\sum_{j=1}^{\infty}{|\lambda_j|^p}\Big)^{{1}/{p}}\leq C_{d,p,q}\big\Vert   b\big\Vert_{f_{p}^{s,q}}.
\end{equation*}
Moreoever, 
\begin{equation*}
\big\Vert  b  \big\Vert_{f_p^{s,q}}\approx \inf{\Big\{ \Big(\sum_{j=1}^{\infty}{|\lambda_j|^p}\Big)^{{1}/{p}}   :  b=\sum_{j=1}^{\infty}{\lambda_j r_j} ,~ r_j ~\text{is an $\infty$-atom for $f_p^{s,q}$}    \Big\}}.
\end{equation*}

\end{lemma}

The proof of the above lemma can be found in \cite[Chapter 7]{Fr_Ja}, \cite{Fr_Ja2}, and \cite[Chapter 2.3.4]{Gr}.

\subsection{Characterization of a vector-valued function space by using a sharp maximal function for $0<q<p<\infty$}

Given a locally integrable function $f$ on $\mathbb{R}^d$ the Fefferman-Stein sharp maximal function $f^{\sharp}$ is defined by 
\begin{equation*}
f^{\sharp}(x):=\sup_{Q:x\in Q}\frac{1}{|Q|}\int_Q{|f(y)-f_Q|}dy
\end{equation*} where $f_Q:=\frac{1}{|Q|}\int_Q{f(z)}dz$ and the supremum is taken over all cubes $Q$ containing $x$. Then a fundamental inequality of Fefferman and Stein \cite{Fe_St1} says that 
for $1<p<\infty$, $1\leq p_0\leq p$ if $f\in L^{p_0}(\mathbb{R}^d)$  then we have 
\begin{equation*}
\Vert \mathcal{M}f\Vert_{L^p}\lesssim_p \Vert f^{\sharp}\Vert_{L^p}.
\end{equation*}

By following the proof of the above estimate in \cite{Fe_St1} one can actually replace the maximal functions by dyadic maximal ones.
For locally integrable function $f$ we define
 the dyadic maximal function 
\begin{equation*}
\mathcal{M}^{(d)}f(x):=\sup_{Q\in \mathcal{D},x\in Q}{\frac{1}{|Q|}\int_Q{|f(y)|}dy},
\end{equation*} 
and the dyadic sharp maximal funtion
\begin{equation*}
\mathcal{M}^{\sharp}f(x):= \sup_{Q\in \mathcal{D}:x\in Q }{\frac{1}{|Q|}\int_Q{|f(y)-f_Q|}dy}
\end{equation*} where the supremums are taken over all dyadic cubes $Q$ containing $x$.
Then for $1<p<\infty$, $1\leq p_0\leq p$, and $f\in L^{p_0}$ one has
\begin{equation}\label{sharpmaximalin}
\Vert \mathcal{M}^{(d)}f\Vert_{L^p}\lesssim_p \Vert \mathcal{M}^{\sharp}f\Vert_{L^p}.
\end{equation}

The next lemma states a pointwise estimate of sharp maximal functions, which is a slight modification of Lemma 6.4 in \cite{Se2}.
For $n\in\mathbb{N}$ and a sequence of functions $\{g_k\}_{k\in\mathbb{N}}$ let
\begin{equation*}
\mathcal{N}_{q}^{\sharp,n}\big( \{g_k\}_{k\in\mathbb{N}}\big)(x):=\sup_{P:x\in P\in\mathcal{D}}{\Big(\dfrac{1}{|P|}\int_P{\sum_{k=\max(n,-\log_2{l(P)})}^{\infty} |g_k(y)|^q}dy \Big)^{1/q}}.
\end{equation*}
\begin{lemma}\label{lemmasharp}
Let $0<q<\infty$, $\sigma>2d/q$, and $n\in\mathbb{N}$. Suppose $g_k\in\mathcal{E}(A2^{k})$ for each $k\in\mathbb{N}$. Then 
\begin{equation*}
\mathcal{N}_q^{\sharp,n}\big(\{\mathfrak{M}_{\sigma,2^k}{g_k}\}_{k\in\mathbb{N}}\big)(x)\lesssim_{\sigma,q,A}  \mathcal{N}_q^{\sharp,n}\big(\{{g_k}\}_{k\in\mathbb{N}}\big)(x).
\end{equation*}

\end{lemma}

\begin{proof}
We may assume $A=1$ without loss of generality.
We claim that for each $k\in\mathbb{N}$, $P\in\mathcal{D}$ with $l(P)\geq 2^{-k}$, and any $t>0$
\begin{equation}\label{newclaim}
\mathfrak{M}_{\sigma,2^k}g_k(y)\lesssim_{\sigma,t} \sum_{l=0}^{\infty}{2^{-l(\sigma-d/t)}\mathcal{M}_t\big( \chi_{2^{l+3}P}g_k\big)(y)}, \quad y\in P
\end{equation} where $2^{l+3}P$ stands for a dilate of $P$ by a factor of $2^{l+3}$ with the same center.
Once we have (\ref{newclaim}), by choosing  $0<t<q$ so that $\sigma>d/t+d/q>2d/q$, it follows that 
\begin{align*}
&\mathcal{N}_q^{\sharp,n}\big(\{\mathfrak{M}_{\sigma,2^k}{g_k}\}\big)(x)\\
&\lesssim\sup_{P:x\in P\in\mathcal{D}}{\Big(\dfrac{1}{|P|}\int_P\sum_{k=\max(n,-\log_2{l(P)})}^{\infty}{\Big(\sum_{l=0}^{\infty}{2^{-l(\sigma-d/t)}\mathcal{M}_t\big(\chi_{2^{l+3}P}g_k\big)(y)}  \Big)^q}dy \Big)^{1/q}}\\
&\lesssim \sup_{P:x\in P\in\mathcal{D}}{\Big(\sum_{l=0}^{\infty}{2^{-lq\epsilon(\sigma-d/t)}\dfrac{1}{|P|}\sum_{k=\max(n,-\log_2{l(P)})}^{\infty}{\big\Vert \mathcal{M}_t(\chi_{2^{l+3}P}g_k)  \big\Vert_{L^q}^q}   } \Big)^{1/q}}\\
&\lesssim \Big(\sum_{l=0}^{\infty}{2^{-lq(\epsilon(\sigma-d/t)-d/q)}} \Big)^{p/q}\sup_{R:x\in R\in\mathcal{D}}\Big(\dfrac{1}{|R|}\int_R{\sum_{k=\max(n,-\log_2{l(R)})}^{\infty}{|g_k(y)|^q}}dy \Big)^{1/q}\\
&\lesssim \mathcal{N}_q^{\sharp,n}\big(\{{g_k}\}\big)(x)
\end{align*}
for $0<\epsilon<1$ satisfying $\epsilon(\sigma-d/t)>d/q$ where the second inequality follows from $l^q\subset l^1$ if $q\leq 1$ or from H\"older's inequality if $q>1$, and the third one follows from the $L^q$-boundedness of $\mathcal{M}_t$ and the fact that $-\log_2{l(2^{l+3}P)}\leq -\log_2{l(P)}$.

Let us prove (\ref{newclaim}). Let $P\in\mathcal{D}$ with $l(P)\geq 2^{-k}$ and $y\in P$.
By using Peetre's mean value inequality in \cite{Pe}
 we see that for all $t>0$ and sufficiently small $\delta>0$
\begin{equation*}
\mathfrak{M}_{\sigma,2^k}g_k(y)\lesssim_{\delta,t} \sup_{z\in\mathbb{R}^d}{\dfrac{1}{(1+2^k|z|)^{\sigma}}\Big(\dfrac{1}{2^{-kd}}\int_{|u|<2^{-k}\delta}{\big|g_k(y-z-u) \big|^t}du \Big)^{1/t}}
\end{equation*} and this is bounded by
\begin{equation*} 
  \sum_{l=0}^{\infty}{2^{-l\sigma}\sup_{|z|\leq 2^{-k+l}}{\Big(\dfrac{1}{2^{-kd}}\int_{|y-z-u|<2^{-k}\delta}{\big|g_k(u) \big|^t}du \Big)^{1/t}}}.
\end{equation*}
We observe that the supremum in the sum is less than
\begin{equation*}
2^{ld/t}{\Big( \dfrac{1}{2^{(-k+l)d}}\int_{|y-u|<2^{-k+l+1}}{\big|g_k(u) \big|^t}du\Big)^{1/t}}\lesssim 2^{ld/t}\mathcal{M}_t\big(\chi_{2^{l+3}P}g_k \big)(y)
\end{equation*} for $|z|\leq 2^{-k+l}$ and  this  proves (\ref{newclaim}).
\end{proof}

Now we have the following characterization of a vector-valued function space $L^p(l^q)$ for $0<q<p<\infty$, which is an inhomogeneous analogue of \cite[Proposition 6.1, 6.2]{Se2}.
\begin{lemma}\label{54}
Let $0<q<p<\infty$, $n\in\mathbb{N}$, and $A>0$. Suppose $g_k\in \mathcal{E}(A2^k)$ for each $k\geq n$. Then
\begin{equation*}
\big\Vert \big\{ g_k\big\}_{k=n}^{\infty}\big\Vert_{L^p(l^q)}\approx_{n,A} \big\Vert \mathcal{N}_q^{\sharp,n}\big( \{g_k\}\big)  \big\Vert_{L^p}.
\end{equation*}

\end{lemma}

\begin{proof}
Since \begin{equation*}
\mathcal{N}_q^{\sharp,n}\big(\{g_k\}\big)(x)\lesssim \mathcal{M}_q\big(\Vert \{g_k\}\Vert_{l^q} \big)(x)
\end{equation*} the inequality $``\gtrsim"$ follows from the $L^p$ boundedness of $\mathcal{M}_q$.

For the opposite direction, we apply  (\ref{sharpmaximalin}) with $p/q>1$ and then
\begin{equation}\label{basicestt}
\big\Vert \big\{ g_k\big\}_{k=n}^{\infty}\big\Vert_{L^p(l^q)}\lesssim \Big\Vert \mathcal{M}^{\sharp}\Big(\sum_{k=n}^{\infty}{| g_k|^q} \Big)\Big\Vert_{L^{p/q}}^{1/q}.
\end{equation}
We see that
\begin{align}
\mathcal{M}^{\sharp}\Big(\sum_{k=n}^{\infty}{|g_k|^q}\Big)(x)
&\leq \sup_{x\in P\in\mathcal{D}}{\Big(\frac{1}{|P|}\int_P\frac{1}{|P|}\int_P{\sum_{k=n}^{\infty}{\big|g_k(y)-g_k(z)\big|^q}}dzdy \Big)}\nonumber\\
&\lesssim \big(\mathcal{N}_{q}^{\sharp,n}\big(\{g_k\}\big)(x) \big)^{q}\nonumber\\
&\relphantom{=}+ \sup_{x\in P\in\mathcal{D}}{\Big(\frac{1}{|P|}\int_P\frac{1}{|P|}\int_P{\sum_{k=n}^{-\log_2{l(P)-1}}{\big|g_k(y)-g_k(z)\big|^q}}dzdy \Big)} \label{smallkterm}.
\end{align}
If $l(P)\leq 2^{-k-1}$ then there exists the unique dyadic cube $Q_P\in\mathcal{D}_{k+1}$ containing $P$. Then, by using Taylor's formula, one obtains that
(\ref{smallkterm}) is less than a constant times
\begin{equation*}
\sup_{x\in P\in\mathcal{D}}{\Big( \sum_{k=n}^{-\log_2{l(P)}-1}{\big( 2^kl(P)\big)^q\big( \sup_{w\in Q_P}{|\psi_k|\ast |g_k|(w)}\big)^q}\Big)}
\end{equation*}
for some $\psi_k\in S$ with $Supp(\widehat{\psi_k})\subset \big\{\xi\in\mathbb{R}^d:|\xi|\lesssim 2^k \big\}$.
Observe that for any $\sigma>0$ 
\begin{align*}
\sup_{w\in Q_P}{|\psi_k|\ast |g_k|(w)}&\lesssim_{\sigma}\inf_{w\in Q_P}{\mathfrak{M}_{\sigma,2^k}\big(|\psi_k|\ast |g_k|\big)(w)}\\
 &\lesssim \inf_{w\in Q_P}{\mathfrak{M}_{\sigma,2^k}\big( \mathfrak{M}_{\sigma,2^k}g_k\big)(w)}\lesssim \inf_{w\in Q_P}{\mathfrak{M}_{\sigma,2^k}g_k(w)}
\end{align*}
and this yields that 
\begin{align*}
(\ref{smallkterm})&\lesssim \sup_{x\in P\in\mathcal{D}}\Big( \sum_{k=n}^{-\log_2{l(P)}-1}{\big( 2^kl(P)\big)^q\Big(\inf_{w\in Q_P}{\mathfrak{M}_{\sigma,2^k}g_k(w)}\Big)^q}\Big)\\
 &\lesssim \sup_{x\in P\in\mathcal{D}}{\sup_{k\geq n}{\inf_{w\in Q_P}{\big(\mathfrak{M}_{\sigma,2^k}g_k(w) \big)^q}}}\\
  &\lesssim \sup_{x\in P\in\mathcal{D}}{\sup_{k\geq n}{\Big(\frac{1}{|Q_P|}\int_{Q_P}{\big( \mathfrak{M}_{\sigma,2^k}g_k(w)\big)^q}dw \Big)}}\\
  &\lesssim \mathcal{N}_q^{\sharp,n}\big( \{\mathfrak{M}_{\sigma,2^k}g_k \}\big)(x).
\end{align*}
By using Lemma \ref{lemmasharp} with $\sigma>2d/q$ one obatins
\begin{equation*}
\mathcal{M}^{\sharp}\Big(\sum_{k=n}^{\infty}{|g_k|^q}\Big)(x)\lesssim \big(\mathcal{N}_{q}^{\sharp,n}\big(\{g_k\}\big)(x) \big)^{q},
\end{equation*}
which completes the proof with (\ref{basicestt}).
\end{proof}


\section{Proof of Theorem \ref{main}}\label{proofmain}

The proof is based on the paradifferential technique as in \cite{Park}.
Let \begin{equation*}
a_{j,k}(x,\xi)=
\begin{cases}
\phi_j\ast a(\cdot,\xi)(x)\widehat{\phi_k}(\xi) \quad &
\quad \quad j,k\geq 0
\\
0 \quad &
\quad  otherwise.
\end{cases}
\end{equation*}
Write
\begin{align*}
a(x,\xi)    &=\sum_{j=3}^{\infty}{\sum_{k=0}^{j-3}{a_{j,k}(x,\xi)}}+\sum_{k=0}^{\infty}{\sum_{j=k-2}^{k+2}{a_{j,k}(x,\xi)}}+\sum_{k=3}^{\infty}{\sum_{j=0}^{k-3}{a_{j,k}(x,\xi)}}\\
    &=:a^{(1)}(x,\xi)+a^{(2)}(x,\xi)+a^{(3)}(x,\xi).
\end{align*}
Note that $a^{(j)}\in \mathcal{S}_{0,0}^{m}$ for each $j=1,2,3$.

It was already proved in \cite{Park} that for any $s,m\in\mathbb{R}$ and $0<p,t\leq\infty$  we have
\begin{equation}\label{strongerest}
\Vert T_{[a^{(j)}]}f\Vert_{F_p^{0,t}}\lesssim \Vert f\Vert_{F_p^{s,t}}, \quad j=1,2,
\end{equation} 
which clearly implies 
\begin{equation*}
T_{[a^{(j)}]}:F_p^{s_1,q}\to F_{p}^{s_2,t}, \quad j=1,2
\end{equation*} for all $0<q,t\leq \infty$ and $s_1,s_2\in\mathbb{R}$. 
For the sake of completeness, a sketch of those bounds is provided in what follows. We start with the case $j=1$.
Note that the Fourier transform of $\sum_{k=0}^{j-3}{T_{[a_{j,k}]}f}$ is supported in an annulus of size $2^j$. Then the technique of Nikol'skii representation in \cite[2.5.2]{Tr} with the Fourier support condition yields that
\begin{equation}\label{aa11}
\big\Vert T_{[a^{(1)}]}f\big\Vert_{F_p^{0,t}}\lesssim \Big\Vert \Big( \sum_{j=3}^{\infty}{\Big| \sum_{k=0}^{j-3}{T_{[a_{j,k}]}{f}}\Big|^t}\Big)^{1/t}\Big\Vert_{L^p}.
\end{equation} 
Now using Marschall's inequality, stated in \cite[Lemma3.2]{Park}, one obtains that for any $N>0$
\begin{equation}\label{aa111}
\big| T_{[a_{j,k}]}f(x)\big|\lesssim_{N}2^{k(m+1+d/r)}2^{-jN}\mathcal{M}_r\big(\widetilde{\Lambda_k} f\big)(x)
\end{equation}
where $\widetilde{\Lambda_k}$ is a convolution operator having a similar property of $\Lambda_k$ with $\widetilde{\Lambda_k}\Lambda_k=\Lambda_k$.
Inserting (\ref{aa111}) in (\ref{aa11}) and using the Fefferman-Stein vector-valued maximal inequality,
\begin{equation*}
\big\Vert T_{[a^{(1)}]}f\big\Vert_{F_p^{0,t}}\lesssim_N \Vert f\Vert_{F_p^{c_{(m,d,r)}-N,t}}
\end{equation*} for some constant $c_{(m,d,r)}$, depending only on $m$, $d$, $r$,
and the required conclusion follows by letting $N$ sufficiently large and using embedding theorem.
To estimate the case $j=2$, we rewrite
\begin{equation*}
T_{[a^{(2)}]}f=\sum_{k=0}^{\infty}{\sum_{j=k-2}^{j+2}{T_{[a_{j,k}]}f}}=:\sum_{k=0}^{\infty}{T_{[a_k]}f}.
\end{equation*}
Then the kernel $K_k(x,y)$ of $T_{[a_k]}$ has the size estimate that for any $M,J>0$
\begin{equation*}
|K_k(x,y)|\lesssim_{J,M}2^{-Jk}\frac{1}{(1+|x-y|)^M},
\end{equation*} which follows from the method of repeated integrations by parts in \cite{St}.
Consequently, choosing $\sigma>d/\min{(p,t)}$ and  $M>\sigma+d$,
\begin{equation*}
\big| \Lambda_jT_{[a_k]}f(x)\big|\lesssim_{M,\sigma} 2^{-k(J-\sigma)}\mathfrak{M}_{\sigma,2^k}\big(\widetilde{\Lambda_k}f\big)(x) \quad \text{ uniformly in }~ j\geq 0
\end{equation*}
and finally, we apply the maximal inequalities (\ref{max}) and (\ref{inftymaximal}) to conclude that
\begin{equation*}
\Vert T_{[a^{(2)}]}f\Vert_{F_p^{0,t}}\lesssim \Vert f\Vert_{F_p^{c_{(p,t)}-J,t}}
\end{equation*}
for some constant $c_{(p,t)}$, depending on $p$ and $t$.
Then (\ref{strongerest}) clearly follows by letting $J$ large enough and using embedding theorem.

We now turn to the estimate for $T_{[a^{(3)}]}$.
As mentioned in Section \ref{intro}, it suffices to show that if $m-s_1+s_2=-d\big|1/2-1/p \big|$ then
\begin{align}
\big\Vert T_{[a^{(3)}]}f\big\Vert_{F_p^{s_2,p}(\mathbb{R}^d)}&\lesssim \Vert f\Vert_{F_p^{s_1,\infty}(\mathbb{R}^d)} \quad \text{ for }~ 0<p\leq 1  \label{bigclaim11} \\
\big\Vert T_{[a^{(3)}]}f\big\Vert_{F_p^{s_2,t}(\mathbb{R}^d)}&\lesssim \Vert f\Vert_{F_p^{s_1,p}(\mathbb{R}^d)}  \quad \text{ for }~ 2<p\leq \infty,~ 0<t<1. \label{bigclaim22}
\end{align}

Note that $T_{[a^{(3)}]}$ can be written as 
\begin{equation*}
T_{[a^{(3)}]}=\sum_{k=3}^{\infty}{T_{[b_k]}}
\end{equation*} where $b_k(x,\xi):=\big( \sum_{j=0}^{k-3}{\phi_j}\big)\ast a(\cdot,\xi)(x)\widehat{\phi_k}(\xi)$ is also a $\mathcal{S}_{0,0}^{m}$ symbol with a constant which is independent of $k$.
We first observe that for $0<r<\infty$  there exists a constant $C_r>0$ such that
\begin{equation}\label{previous}
\big\Vert T_{[b_k]}g_k\big\Vert_{L^r}\leq C_r  2^{k(m+d|1/2-1/r|)}\Vert g_k\Vert_{L^r},
\end{equation}
provided $g_k\in S$ satisfies $\widehat{g_k}\subset \{|\xi|\approx 2^k\}$ for each $k\in\mathbb{N}$.
This follows from the $h^r$ boundedness of $T_{[b_k]}$ and $\Vert g_k\Vert_{h^r}\approx_r \Vert g_k\Vert_{L^r}$ with the Fourier support condition of $g_k$.
Moreover, for $r=\infty$ we claim that for each $k\geq 3$ if $g_k\in C^{\infty}$ satisfies the polynomial growth estimate
\begin{equation}\label{polygrowthk}
|g_k(y)|\leq C_k (1+|y|)^{N_k}
\end{equation} for some $C_k$, $N_k>0$ then
\begin{equation}\label{inftyneed}
\big\Vert T_{[b_k]}g_k\big\Vert_{L^{\infty}}\lesssim 2^{k(m+d/2)}\Vert g_k\Vert_{L^{\infty}}
\end{equation} where the constant in the inequality is independent of $k$. 
Note that when $a\in \mathcal{S}_{\rho,\delta}^{m}$, $0\leq \delta<\rho<1$, (\ref{inftyneed}) also holds due to Fefferman \cite{Fe}.
To prove (\ref{inftyneed}) we need the following lemmas.

\begin{lemma}\label{fourierseries}
Suppose $g\in C^{\infty}$ satisfies the polynomial growth estimate
\begin{equation*}
|g(y)|\lesssim \big(1+|y| \big)^{N}, \quad \forall y\in\mathbb{R}^d
\end{equation*} for some $N>0$.
Then for any $\Psi\in S$ and  $x\in \mathbb{R}^d$
\begin{equation}\label{fourierco}
\Big(\sum_{l\in\mathbb{Z}^d}{\big| g\ast \big(\Psi e^{2\pi i\langle \cdot,l\rangle}\big)(x)\big|^2} \Big)^{1/2}=\Big( \int_{[0,1]^d}{\Big| \sum_{n\in\mathbb{Z}^d}{g(x-y+n)\Psi(y-n)}\Big|^2}dy\Big)^{1/2}.
\end{equation}
\end{lemma}

\begin{proof}
For each $x\in\mathbb{R}^d$ we define 
\begin{equation*}
G_x(y):=\sum_{n\in\mathbb{Z}^d}{g(x-y+n)\Psi(y-n)}.
\end{equation*}
We see that each $G_x$ is well-defined periodic function and $|G_x(y)|\lesssim_N \big(1+|x|\big)^{N}$.
Then (\ref{fourierco}) follows from Plancherel's identity and the observation
\begin{align*}
g\ast \big(\Psi e^{2\pi i\langle \cdot,l\rangle}\big)(x)&=\sum_{n\in\mathbb{Z}^d}\int_{-n+[0,1]^d}{g(x-y)\Psi(y)e^{2\pi i\langle y,l\rangle}}dy\\
&=\int_{[0,1]^d}{\sum_{n\in\mathbb{Z}^d}{g(x-y+n)\Psi(y-n)}e^{2\pi i\langle y,l\rangle}}dy\\
&=\int_{[0,1]^d}{G_x(y)e^{2\pi i\langle y,l\rangle}}dy.
\end{align*}  
\end{proof}

\begin{lemma}\label{inftylemma}
Let $\mu\in \mathbb{N}$ and $P\in\mathcal{D}_{\mu}$. Suppose $\{g_k\}_{k\in\mathbb{N}}$ is a sequence of $C^{\infty}$ functions satisfying $(\ref{polygrowthk})$.
Then there exists $\epsilon>0$ such that
\begin{equation}\label{inftyest}
\Big(\frac{1}{|P|}\int_P{\big|T_{[b_{k}]}g_k(x)\big|^2}dx \Big)^{1/2}\lesssim_{\epsilon}2^{k(m+d/2)}2^{-\epsilon(k-\mu)}\Vert g_k\Vert_{L^{\infty}}
\end{equation} for each integer $k\geq \mu$. Here the constant in the inequality is independent of $k$, $\mu$ and $P$.

\end{lemma}

\begin{proof}
Choose $\psi, \widetilde{\psi}\in S$ so that $Supp(\widehat{\psi})\subset \{\xi:|\xi|\leq 1+1/100\}$, $\widehat{\psi}(\xi)=1$ on $\{\xi:|\xi|\leq 1-1/100\}$, $\sum_{l\in\mathbb{Z}^d}{\widehat{\psi}(\xi-l)}=1$ for all $\xi\in\mathbb{R}^d$, $\widetilde{\psi}\ast\psi=\psi$, and $Supp(\widehat{\widetilde{\psi}})\subset \{\xi:|\xi|\leq 1+1/10\}$. 
For each $l\in\mathbb{Z}^d$ let $b_{k}^{l}(x,\xi):= b_{k}(x,\xi)\widehat{\widetilde{\psi}}(\xi-l)$.

Fix $k\geq \mu$ and for each $l\in \mathbb{Z}^d$ let $\widehat{g_k^l}(\xi):=\widehat{g_k}(\xi)\widehat{{\psi}}(\xi-l)$.
Let $P^*$ be a dilate of $P\in\mathcal{D}_{\mu}$ by a factor of $10$, and for $0<\delta<1$ and $k\geq \mu$ let $P_{\mu,k}^{\delta}$ be a dilate of $P\in\mathcal{D}_{\mu}$ whose side length is $2^{\delta(k-\mu)}$ with the same center $c_P$.
Then the left-hand side of (\ref{inftyest}) is
\begin{equation*}
\Big(\frac{1}{|P|}\int_P{\Big|\sum_{|l|\approx 2^k}{T_{[b_{k}^{l}]}g_{k}^{l}(x)} \Big|^2}dx \Big)^{1/2}\leq I_P+II_P
\end{equation*} where
\begin{equation*}
I_P:=\Big(\frac{1}{|P|}\int_P{\Big|\sum_{|l|\approx 2^k}{T_{[b_{k}^{l}]} \big(\chi_{P_{\mu,k}^{\delta}}g_k^l\big)(x)} \Big|^2}dx \Big)^{1/2}
\end{equation*}
\begin{equation*}
II_P:=\Big(\frac{1}{|P|}\int_P{\Big|\sum_{|l|\approx 2^k}{T_{[b_{k}^{l}]} \big(\chi_{(P_{k,\mu}^{\delta})^c}g_k^l\big)(x)} \Big|^2}dx \Big)^{1/2}.
\end{equation*}

By using  (\ref{previous})  we have
\begin{align*}
I_P&\leq 2^{\mu d/2}\Big\Vert \sum_{|l|\approx 2^k}{T_{[b_{k}^{l}]} \big(\chi_{P_{\mu,k}^{\delta}}g_k^l\big)}\Big\Vert_{L^2}=2^{\mu d/2}\Big\Vert T_{[b_{k}]}\Big(\sum_{|l|\approx 2^k}{\big(\widetilde{\psi}e^{2\pi i\langle \cdot,l\rangle}\big)\ast \big(\chi_{P_{\mu,k}^{\delta}}g_k^l\big)} \Big)\Big\Vert_{L^2}\\
    &\lesssim2^{k(m+d/2)}2^{-(k-\mu) d/2}\Big\Vert \sum_{|l|\approx 2^k}{\big(\widetilde{\psi}e^{2\pi i\langle \cdot,l\rangle}\big)\ast \big(\chi_{P_{\mu,k}^{\delta}}g_k^l\big)}\Big\Vert_{L^2}
\end{align*}
and the almost orthogonality property of $\big\{\widehat{\widetilde{\psi}}(\cdot-l)\big\}_{l\in\mathbb{Z}^d}$ and Young's inequality yield that the last expression is majored by  
\begin{align*}
 &2^{k(m+d/2)} 2^{-(k-\mu)d/2}\Big( \sum_{|l|\approx 2^k}{\big\Vert \chi_{P_{\mu,k}^{\delta}}g_k^l\big\Vert_{L^2}^2}\Big)^{1/2}\\
 &\leq 2^{k(m+d/2)}2^{-(k-\mu)d/2}\Big(\int_{P_{\mu,k}^{\delta}}{\sum_{|l|\approx 2^k}{\big|g_k\ast \big(\psi e^{2\pi i\langle \cdot,l\rangle}\big)(x) \big|^2}}dx\Big)^{1/2}.
\end{align*}
Then we apply Lemma \ref{fourierseries} to conclude that 
\begin{align*}
I_P&\lesssim 2^{k(m+d/2)}2^{-(k-\mu)d/2}\Big(\int_{P_{\mu,k}^{\delta}}{\int_{[0,1]^d}{\Big|\sum_{n\in\mathbb{Z}^d}{g_k(x-y+n)\psi(y-n)} \Big|^2}dy}dx \Big)^{1/2}\\
&\lesssim 2^{k(m+d/2)}2^{-(k-\mu)(1-\delta)d/2}\Vert g_k\Vert_{L^{\infty}}.
\end{align*}

To estimate $II_P$ let $K_{k}^l(x,y)$ be the kernel of $T_{[b_{k}^{l}]}$. Then it follows from integration by parts that
\begin{equation*}
|K_{k}^{l}(x,y)|\lesssim_M  2^{km}\frac{1}{|x_P-y|^M}, \quad |l|\approx 2^k
\end{equation*}
for $x\in P$ and $y\in (P_{\mu,k}^{\delta})^c$.
Therefore
\begin{align*}
II_P&\leq \sup_{x\in P}{\sum_{|l|\approx 2^k}{\big|T_{[b_{k}^{l}]} \big(\chi_{(P_{\mu,k}^{\delta})^c}g_k^l\big)(x) \big|}}\\
&\lesssim_M 2^{km}\int_{|x_P-y|\gtrsim 2^{\delta(k-\mu)}}{\frac{1}{|x_P-y|^M}\sum_{|l|\approx 2^k}{|g_k^l(y)|}}dy\\
&\lesssim 2^{k(m+d/2)}2^{-\delta(k-\mu)(M-d)}\Big\Vert \Big(\sum_{l\in\mathbb{Z}^d}{|g_k^l|^2} \Big)^{1/2}\Big\Vert_{L^{\infty}}
\end{align*} for $M>d$.
Now we apply Lemma \ref{fourierseries} to obtain
\begin{align*}
\Big\Vert \Big(\sum_{l\in\mathbb{Z}^d}{|g_k^l|^2} \Big)^{1/2}\Big\Vert_{L^{\infty}}&=\Big\Vert \Big(\sum_{l\in\mathbb{Z}^d}{|g_k\ast\big(\psi e^{2\pi i\langle \cdot,l\rangle}\big)|^2} \Big)^{1/2}\Big\Vert_{L^{\infty}}\\
&\leq \Vert g_k\Vert_{L^{\infty}}\Big\Vert \sum_{n\in\mathbb{Z}^d}{|\psi(\cdot-n)|} \Big\Vert_{L^{\infty}}\lesssim \Vert g_k\Vert_{L^{\infty}}.
\end{align*}
This yields that for $M>d$
\begin{equation*}
II_P\lesssim_M 2^{k(m+d/2)}2^{-\delta(k-\mu)(M-d)}\Vert g_k\Vert_{L^{\infty}}.
\end{equation*} 

The proof is done by choosing $\epsilon=\min{\big((1-\delta)d/2,\delta(M-d)\big)}$.
\end{proof}

Now we return to the proof of (\ref{inftyneed}).
By using (\ref{inftymaximal}) and Lemma \ref{inftylemma}, it follows that 
\begin{align*}
\big\Vert T_{[b_k]}g_k\big\Vert_{L^{\infty}}&\lesssim \big\Vert T_{[b_k]}g_k\big\Vert_{F_{\infty}^{0,\infty}}\lesssim \big\Vert T_{[b_k]}g_k\big\Vert_{F_{\infty}^{0,2}}\\
&\leq \sup_{P\in\mathcal{D}, 2^{-k-2}\leq l(P)<1}{\Big( \frac{1}{|P|}\int_{P}{\sum_{j=k-2}^{k+2}{\big|\Lambda_jT_{[b_k]}g_k(x) \big|^2}}dx\Big)^{1/2}}\\
&\lesssim_{\sigma} \sup_{P\in\mathcal{D}, 2^{-k-2}\leq l(P)<1}{\Big( \frac{1}{|P|}\int_{P}{{\big(\mathfrak{M}_{\sigma,2^k}T_{[b_k]}g_k(x) \big)^2}}dx\Big)^{1/2}}\\
&\lesssim \sup_{P\in\mathcal{D}, 2^{-k-2}\leq l(P)<1}{\Big( \frac{1}{|P|}\int_{P}{{\big|T_{[b_k]}g_k(x) \big|^2}}dx\Big)^{1/2}}\\
&\lesssim_{\epsilon} 2^{k(m+d/2)}\Vert g_k\Vert_{L^{\infty}}\sup_{1\leq \mu\leq k+2}{2^{-\epsilon(k-\mu)}} \lesssim 2^{k(m+d/2)}\Vert g_k\Vert_{L^{\infty}}
\end{align*} for $\sigma>d/2$ because $Supp(\widehat{T_{[b_k]}g_k})\subset \{\xi:2^{k-2}\leq |\xi|\leq 2^{k+2}\}$.

\subsection{Proof of (\ref{bigclaim11})}
Suppose $0<p\leq 1$ and $m-s_1+s_2=-d(1/p-1/2)$.
According to (\ref{decomposition1}) and Lemma \ref{decomhardy}, $f\in F_p^{s_1,\infty}$ can be decomposed with $\{b_Q\}_{\substack{Q\in\mathcal{D}\\l(Q)\leq 1}}\in f_p^{s_1,\infty}$ and there exist a sequence of scalars $\{\lambda_j\}$ and a sequence of $\infty$-atoms $\{r_{j,Q}\}$ for $f_p^{s_1,\infty}$ such that
\begin{equation*}
f(x)=\sum_{Q\in\mathcal{D}, l(Q)\leq 1}{b_Q\vartheta^Q(x)}=\sum_{j=1}^{\infty}{\lambda_j \sum_{Q\in\mathcal{D}, l(Q)\leq 1}{ r_{j,Q}\vartheta^Q (x)  }}.
\end{equation*}
Then by using the same arguments in \cite[(3.16)]{Park}, one has
\begin{equation*}
\big\Vert T_{[a^{(3)}]}f\big\Vert_{F_p^{s_2,p}}\lesssim \Big(\sum_{j=1}^{\infty}{|\lambda_j|^p}\Big)^{1/p}\sup_j\Big(\sum_{k=3}^{\infty}{2^{s_2kp}\Big\Vert T_{[b_k]}\Big(\sum_{Q\in\mathcal{D},l(Q)\leq 1}{r_{j,Q}\vartheta^{Q}}\Big)\Big\Vert_{L^p}^p} \Big)^{1/p}.
\end{equation*}
Since $\big(\sum_{j=1}^{\infty}{|\lambda_j|^p}\big)^{1/p}\lesssim \Vert f\Vert_{F_p^{s_1,\infty}}$
we are reduced to proving that
 \begin{equation}\label{reduction}
\Big( \sum_{k=\max{(3,\mu)}}^{\infty}{2^{s_2kp}\big\Vert T_{[b_k]}R_{P,k}\big\Vert_{L^p}^p}\Big)^{1/p}\lesssim 1, \quad \text{uniformly in }\mu \in \mathbb{Z} \text{ and }  P \in \mathcal{D}_{\mu}
\end{equation}
where $\{r_Q\}$ is $\infty$-atoms for $f_p^{s_1,\infty}$ associated with $P$, and
 \begin{equation*}
 R_{P,k}(x):=\sum_{\substack{Q\in\mathcal{D}_k,Q\subset P, \\ l(Q)\leq 1}}{r_Q\vartheta^Q(x)}.
 \end{equation*}
The sum in (\ref{reduction}) was additionally taken over $\mu\leq k$ because the condition $Q\subset P$ in the definition of $R_{P,k}$ ensures that $R_{P,k}$ vanishes unless $\mu\leq k$.

We first consider the case $l(P)\leq 2^{-3}$ ( i.e. $\mu\geq 3$ ).
Our claim is that there exists $\epsilon>0$ such that
\begin{equation}\label{reductionshow}
\big\Vert T_{[b_k]}R_{P,k}\big\Vert_{L^p}\lesssim 2^{-s_2k}2^{-\epsilon (k-\mu)} \quad \text{uniformly in }\mu \text{ and } P 
\end{equation} for each $k\geq \mu$.
Then  (\ref{reduction}) follows immediately.
To show (\ref{reductionshow}) we fix $0<\delta<1$ and let $P^*$ and $P_{\mu,k}^{\delta}$ be dilates of $P$ as in the proof of Lemma \ref{inftylemma}.
The left-hand side of (\ref{reductionshow}) is less than a constant times
\begin{equation*}
 \big\Vert T_{[b_k]} R_{P,k}\big\Vert_{L^p(P_{\mu,k}^{\delta})}+\big\Vert T_{[b_k]}\big(\chi_{P^*}R_{P,k}\big)\big\Vert_{L^p((P_{\mu,k}^{\delta})^c)}+\big\Vert T_{[b_k]}\big(\chi_{(P^*)^c}R_{P,k}\big)\big\Vert_{L^p((P_{\mu,k}^{\delta})^c)}.
\end{equation*}
By applying H\"older's inequality and (\ref{previous}) the first term is majored by
\begin{equation*}
 |P_{\mu,k}^{\delta}|^{1/p-1/2}\big\Vert T_{[b_k]} R_{P,k}\big\Vert_{L^2}\lesssim 2^{\delta(k-\mu)d(1/p-1/2)}2^{km}\Vert R_{P,k}\Vert_{L^2}.
\end{equation*}
Since $\Vert R_{P,k}\Vert_{L^2}\approx 2^{\mu d(1/p-1/2)}2^{-s_1k}$,
one has
\begin{equation*}
\big\Vert T_{[b_k]} R_{P,k}\big\Vert_{L^p(P_{\mu,k}^{\delta})}\lesssim 2^{-s_2k} 2^{-(k-\mu)d(1/p-1/2)(1-\delta)}.
\end{equation*}
Moreover, we observe that $T_{[b_k]} \big(\chi_{(P^*)^c}R_{P,k}\big)=T_{[b_k]} \big(\widetilde{\phi_k}\ast   (\chi_{(P^*)^c}R_{P,k})\big)$ where $\widetilde{\phi_k}:=\phi_{k-1}+\phi_k+\phi_{k+1}$ for $k\geq 3$. Then it follows from (\ref{previous}) that
\begin{equation*}
\big\Vert T_{[b_k]} \big(\chi_{(P^*)^c}R_{P,k}\big)\big\Vert_{L^p((P_{\mu,k}^{\delta})^c)}\lesssim 2^{(s_1-s_2)k}\big\Vert \widetilde{\phi_k}\ast \big(\chi_{(P^*)^c}R_{P,k}\big)\big\Vert_{L^p}
\end{equation*}
and the argument in \cite[p559]{Park} yields $\big\Vert \widetilde{\phi_k}\ast \big(\chi_{(P^*)^c}R_{P,k}\big)\big\Vert_{L^p}\lesssim_{\epsilon_0} 2^{-s_1k} 2^{-(k-\mu)\epsilon_0}$ for some $\epsilon_0>0$, which establishes
\begin{equation*}
\big\Vert T_{[b_k]} \big(\chi_{(P^*)^c}R_{P,k}\big)\big\Vert_{L^p((P_{\mu,k}^{\delta})^c)} \lesssim 2^{-s_2k }2^{-(k-\mu)\epsilon_0}.
\end{equation*}
Now let us look at the term $\big\Vert T_{[b_k]} \big(\chi_{P^*}R_{P,k}\big)\big\Vert_{L^p((P_{\mu,k}^{\delta})^c)}$.
Let $K_k(x,y)$ be the kernel of $T_{[b_k]}$ and define
\begin{equation*}
c_k(y,\eta):= \int_{\mathbb{R}^d}{K_k(x+y,y)e^{-2\pi i\langle x,\eta\rangle}}dx.
\end{equation*}
Then $\overline{c_k(y,\eta)}$ can be interpreted as a symbol corresponding to the adjoint operator of $T_{[b_k]}$ and therefore $c_k$ also belongs to $\mathcal{S}_{0,0}^{m}$. See \cite[Appendix]{Park} for more details.
 Furthermore $\eta$ lives in the annulus $\{\eta: 2^{k-2}\leq |\eta|\leq 2^{k+2}\}$. Thus for any multi-indices $\alpha$ one has
\begin{equation}\label{kernelest}
\Big(\int_{\mathbb{R}^d}{\big|(x-y)^{\alpha}K_k(x,y) \big|^2}dx \Big)^{1/2}=\Big(\int_{\mathbb{R}^d}{\big|\partial_{\eta}^{\alpha}c_k(y,\eta) \big|^2}d\eta \Big)^{1/2}\lesssim 2^{k(m+d/2)}
\end{equation} by using Plancherel's theorem.
When $0<p<1$, H\"older's inequality with $1/p>1$ and Fubini's theorem prove
\begin{align*}
& \big\Vert T_{[b_k]}\big(\chi_{P^*}R_{P,k}\big)\big\Vert_{L^p((P_{\mu,k}^{\delta})^c)}\\
&\lesssim  \Big( \int_{(P_{\mu,k}^{\delta})^c}{\frac{1}{|x-x_P|^{L/(1-p)}}}dx\Big)^{1/p-1}\int_{y\in P^*}{|R_{P,k}(y)|\int_{x\in (P_{\mu,k}^{\delta})^c}{|x-x_P|^{L/p}|K_k(x,y)|}dx}dy\\
&\lesssim_L 2^{-\delta(k-\mu)(L/(1-p)-d)(1/p-1)}\Big( \int_{(P_{\mu,k}^{\delta})^c}{\frac{1}{|x-x_P|^{2N}}}dx\Big)^{1/2}\\
&\relphantom{=}\times\int_{y\in P^*}{|R_{P,k}(y)|\Big(\int_{x\in (P_{\mu,k}^{\delta})^c}{|x-x_P|^{2L/p+2N}\big|K_k(x,y) \big|^2}dx \Big)^{1/2}}dy
\end{align*} where we recall $x_P$ denotes the lower left corner of $P$.
For $x\in (P_{\mu,k}^{\delta})^c$ and $y\in P^*$ we observe taht  $|x-x_P|\lesssim |x-y|$ and thus
\begin{align*}
& \Big(\int_{x\in (P_{\mu,k}^{\delta})^c}{|x-x_P|^{2L/p+2N}\big|K_k(x,y) \big|^2}dx \Big)^{1/2}\\
&\lesssim \Big(\int_{x\in (P_{\mu,k}^{\delta})^c}{|x-y|^{2L/p+2N}\big|K_k(x,y) \big|^2}dx \Big)^{1/2}\lesssim 2^{k(m+d/2)}
\end{align*} due to (\ref{kernelest}). This proves
\begin{align*}
\big\Vert T_{[b_k]}\big(\chi_{P^*}R_{P,k}\big)\big\Vert_{L^p((P_{\mu,k}^{\delta})^c)}&\lesssim_{L,N} 2^{k(m+d/2)}2^{-\delta(k-\mu)(L/(1-p)-d)(1/p-1)}2^{-\delta(k-\mu)(N-d/2)}\Vert R_{P,k}\Vert_{L^1}\\
&\approx 2^{-s_2k}2^{-\delta(k-\mu)(L/(1-p)-d)(1/p-1)}2^{-\delta(k-\mu)(N-d/2)}2^{-(k-\mu)d(1/p-1)}.
\end{align*} 
Similarly, for $p=1$ one can prove
\begin{equation*}
\big\Vert T_{[b_k]}\big(\chi_{P^*}R_{P,k}\big)\big\Vert_{L^1((P_{\mu,k}^{\delta})^c)}\lesssim_N 2^{-s_2k}2^{-\delta(k-\mu)(N-d/2)},
\end{equation*} 
which completes the proof of (\ref{reductionshow}) for $l(P)\leq 2^{-3}$.

When $l(P)>2^{-3}$ ( i.e. $\mu<3$ ) we replace $P_{\mu,k}^{\delta}$ by a dilate of $P$ whose side length is $2^{k-\mu}$ with the same center $c_P$ in the above argument and then obtain 
\begin{equation*}
\big\Vert T_{[b_k]}R_{P,k}\big\Vert_{L^p}\lesssim_{\epsilon} 2^{-s_2k}2^{-\epsilon k}.
\end{equation*} for some $\epsilon>0$.
Then (\ref{reduction}) follows.

\subsection{The proof of (\ref{bigclaim22})}

Suppose $2<p\leq \infty$ and $m-s_1+s_2=-d(1/2-1/p)$. 
For $k\geq 3$ let $\widetilde{\phi_k}:=\phi_{k-1}+\phi_k+\phi_{k+1}$ as before and $f_k:=\widetilde{\phi_k}\ast f$.
Then note that $T_{[b_k]}f=T_{[b_k]}f_k$ and $Supp(\widehat{T_{[b_k]}f_k})\subset \{\xi:2^{k-2}\leq |\xi|\leq 2^{k+2}\}$.

\subsubsection{The case $2<p<\infty$}
Using the method of Nikol'skii representation, as in (\ref{aa11}), and Lemma \ref{54}, one has
\begin{equation*}
\big\Vert T_{[a^{(3)}]}f\big\Vert_{F_p^{s_2,t}}\lesssim \Big\Vert \Big( \sum_{k=3}^{\infty}{2^{s_2kt}\big|T_{[b_k]}f_k \big|^t}\Big)^{1/t}\Big\Vert_{L^p}\lesssim \big\Vert \mathcal{N}_t^{\sharp,3}\big(\big\{2^{s_2k}T_{[b_k]}f_k \big\}\big)\big\Vert_{L^p}
\end{equation*}
Given a sequence $\{T_{[b_k]}f_k\}$ we can choose dyadic cubes $P(x)$ depending measurably on $x$ so that
\begin{equation*}
\mathcal{N}_t^{\sharp,3}\big(\big\{2^{s_2k}T_{[b_k]}f_k \big\}\big)(x)\lesssim \Big(\frac{1}{|P(x)|}\int_{P(x)}{\sum_{k=\max{(3,-\log_2{l(P(x))})}}^{\infty}{2^{s_2kt}|T_{[b_{k}]}f_k(y)|^t}}dy \Big)^{1/t}.
\end{equation*}
For each $x\in \mathbb{R}^d$ let $\mu(x):=-\log_2{l(P(x))}$ and $\lambda(x):=\max{(3,\mu(x))}$.
Then by using H\"older's inequality with $1/t>1$ the last expression is less than a constant multiple of 
\begin{align*}
&\sum_{k=\lambda(x)}^{\infty}{2^{\delta(k-\lambda(x))}2^{s_2k}\frac{1}{|P(x)|}\int_{P(x)}{\big|T_{[b_{k}]} f_k(y) \big|}dy}\\
&=\sum_{n=0}^{\infty}{2^{\delta n}\frac{1}{|P(x)|}\int_{P(x)}{2^{s_2(n+\lambda(x))}\big|T_{[b_{n+\lambda(x)}]}f_{n+\lambda(x)}(y) \big|}dy}
\end{align*}
where $\delta>0$ is any positive number and the equality follows from a change of variables $n=k-\lambda(x)$.
Therefore,
\begin{align*}
\big\Vert \mathcal{N}_t^{\sharp,3}\big(\big\{2^{s_2k}T_{[b_k]}f_k \big\}\big)\big\Vert_{L^p}&\lesssim_{\delta} \Big\Vert \sum_{n=0}^{\infty}{2^{\delta n}\frac{1}{|P(x)|}\int_{P(x)}{2^{s_2(n+\lambda(x))}\big| T_{[b_{n+\lambda(x)}]}f_{n+\lambda(x)}(y) \big|}dy}\Big\Vert_{L^p(x)}\\
&\leq \sum_{n=0}^{\infty}{2^{\delta n}\Big\Vert \frac{1}{|P(x)|}\int_{P(x)}{2^{s_2(n+\lambda(x))}\big|T_{[a_{n+\lambda(x)}]}f_{n+\lambda(x)}(y) \big|}dy\Big\Vert_{L^p(x)}}
\end{align*}
Now the proof is reduced to proving that for $2<p<\infty$ there exists $\epsilon_0>0$ such that
\begin{equation}\label{mainmainclaim}
\Big\Vert \frac{1}{|P(x)|}\int_{P(x)}{2^{s_2(n+\lambda(x))}\big|T_{[a_{n+\lambda(x)}]}f_{n+\lambda(x)}(y) \big|}dy\Big\Vert_{L^p(x)}\lesssim_{\epsilon_0}2^{-\epsilon_0n}\Vert f\Vert_{F_p^{s_1,p}}
\end{equation}
where the constant in the inequality does not depend on the choice of mapping $x\mapsto P(x)$,
because  (\ref{mainmainclaim}), choosing $0<\delta<\epsilon_0$,  implies that
\begin{equation*}
\big\Vert \mathcal{N}_t^{\sharp,3}\big(\big\{2^{s_2k}T_{[b_k]}f_k \big\}\big)\big\Vert_{L^p}\lesssim \Vert f\Vert_{F_p^{s_1,p}}\Big(\sum_{n=0}^{\infty}{2^{-n(\epsilon_0-\delta)}} \Big) \lesssim \Vert f\Vert_{F_p^{s_1,p}}.
\end{equation*}

In the rest of this section we focus on the proof of (\ref{mainmainclaim}) and 
the main idea is a complex interpolation theorem in \cite[Chapter8]{Fr_Ja1}.
For $0\leq Re(z)\leq 1$ we define
\begin{align*}
&\mathfrak{S}_n^{z}(\{f_k\}_{k\in\mathbb{N}})(x)\\
&:= 2^{d(1-z)(n+\lambda(x))/2}2^{(s_2-d/p)(n+\lambda(x))}\frac{1}{|P(x)|}\int_{P(x)}{\omega_{n+\lambda(x)}(x,y)T_{[b_{n+\lambda(x)}]} f_{n+\lambda(x)}(y)}dy.
\end{align*}
where $\omega_k(x,y)$'s are measurable functions satisfying $\sup_{x,y,k}{\big| \omega_{k}(x,y)\big|}\leq 1$.
Then it suffices to prove that for some $\epsilon_0>0$
\begin{equation*}
\big\Vert \mathfrak{S}_n^{\theta}(\{f_k\}_{k\in\mathbb{N}})\big\Vert_{L^p}\lesssim_{\epsilon_0} 2^{-\epsilon_0n}\Vert f\Vert_{F_p^{s_1,p}}, \quad \theta=1-2/p,
\end{equation*}
which follows from interpolating between
\begin{equation*}
\big\Vert  \mathfrak{S}_n^{i\tau}(\{f_k\}_{k\in\mathbb{N}})\big\Vert_{L^2}\lesssim \Vert f\Vert_{F_2^{s_1,2}}, \quad Re(z)=0,
\end{equation*}
\begin{equation*}
\big\Vert  \mathfrak{S}_n^{1+i\tau}(\{f_k\}_{k\in\mathbb{N}})\big\Vert_{L^{\infty}}\lesssim 2^{-\epsilon n}\Vert f\Vert_{F_{\infty}^{s_1,\infty}}, \quad Re(z)=1
\end{equation*} for some $\epsilon>0$.

For $z=i\tau$, $\tau\in\mathbb{R}$ one has
\begin{align*}
\big|\mathfrak{S}_n^{i\tau}(\{f_k\})(x) \big|&\leq  \frac{1}{|P(x)|}\int_{P(x)}{\sup_{k\in\mathbb{N}}{2^{(s_2-d/p+d/2)k}\big|T_{[b_{k}]} f_k(y) \big|}}dy\\
&\leq \mathcal{M}\Big( \Big( \sum_{k\in\mathbb{N}}{ 2^{2(s_2-d/p+d/2)k}\big| T_{[b_{k}]} f_k\big|^2}\Big)^{1/2}\Big)(x)
\end{align*}
 and then  the $L^2$ estimate for $\mathcal{M}$, Fubini, and (\ref{previous}) establish
 \begin{equation*}
 \big\Vert \mathfrak{S}_{n}^{i\tau}(\{f_k\})\big\Vert_{L^2}\lesssim  \Big( \sum_{k=1}^{\infty}{\big\Vert 2^{(s_2-d/p+d/2)k}T_{[b_{k}]} f_k \big\Vert_{L^2}^2}\Big)^{1/2}\lesssim  \Big( \sum_{k=1}^{\infty}{\big\Vert 2^{s_1k}f_k\big\Vert_{L^2}^2}\Big)^{1/2}\lesssim \Vert f\Vert_{F_p^{s_1,2}}.
 \end{equation*}

Moreover, by the Cauchy-Schwarz inequality and Lemma \ref{inftylemma},
one has
\begin{align*}
\big| \mathfrak{S}_n^{1+i\tau}(\{f_k\})(x)\big|&\leq  2^{(s_2-d/p)(n+\lambda(x))}\frac{1}{|P(x)|}\int_{P(x)}{\big|T_{[b_{n+\lambda(x)}]}f_{n+\lambda(x)}(y) \big|}dy\\
&\lesssim_{\epsilon} 2^{s_1(n+\lambda(x))}2^{-\epsilon(n+\lambda(x)-\mu(x))}\Vert f_{n+\lambda(x)}\Vert_{L^{\infty}}\\
&\leq 2^{-\epsilon n}\sup_{k\in\mathbb{N}}{\Vert 2^{s_1k}f_k\Vert_{L^{\infty}}}
\end{align*} for some $\epsilon>0$.
Therefore, 
\begin{equation*}
\big\Vert \mathfrak{S}_n^{1+i\tau}(\{f_k\}_{k\in\mathbb{N}})\big\Vert_{L^{\infty}}\lesssim 2^{-\epsilon n}\sup_{k\in\mathbb{N}}{\Vert 2^{s_1k}f_k\Vert_{L^{\infty}}}\lesssim 2^{-\epsilon n}\Vert f\Vert_{F_{\infty}^{s_1,\infty}}.
\end{equation*}

Here, the inequalities are all independent of $\omega_k$, and thus, by suitably choosing $\omega_k$, the proof of (\ref{mainmainclaim}) is complete.

\subsubsection{The case $p=\infty$}
By considering the support of $\widehat{T_{[b_k]}f_k}$ and using the method of Nikol'skii representation with replacement of (\ref{max}) by (\ref{inftymaximal}), one has
\begin{align*}
\big\Vert T_{[a^{(3)}]}f\big\Vert_{F_{\infty}^{s_2,t}}&=\Big\Vert \sum_{k=3}^{\infty}{T_{[b_k]}f_k}\Big\Vert_{F_{\infty}^{s_2,t}}\\
 &\lesssim \sup_{P\in\mathcal{D},l(P)<1}{\Big( \frac{1}{|P|}\int_P{\sum_{k=\max{(3,-2-\log_2{l(P)})}}^{\infty}{2^{s_2kt}\big| T_{[b_k]}f_k(x)\big|^t}}dx\Big)^{1/t}}\\
 &\lesssim \sup_{P\in\mathcal{D},l(P)\leq 2^{-3}}{\Big( \frac{1}{|P|}\int_P{\sum_{k=-\log_2{l(P)}}^{\infty}{2^{s_2kt}\big| T_{[b_k]}f_k(x)\big|^t}}dx\Big)^{1/t}}.
\end{align*}
Then H\"older's inequality with $2/t>1$ and Lemma \ref{inftylemma} yield that the last expression is majored by a constant times
\begin{equation*}
\sup_{P\in\mathcal{D},l(P)\leq 2^{-3}}\Big(\sum_{k=-\log_2{l(P)}}^{\infty}{2^{s_1kt}\big(2^k l(P)\big)^{-\epsilon t}\Vert f_k\Vert_{L^{\infty}}^t} \Big)^{1/q}\lesssim \Vert f\Vert_{F_{\infty}^{s_1,\infty}}
\end{equation*}  for some $\epsilon>0$.

\section{\textbf{Proof of Theorem \ref{main2}}}\label{proofmain2}
We need to show that (\ref{bigclaim3}) holds for all $0<p,q\leq \infty$ under the assumption $m-s_1+s_2=-d\big|1/2-1/p \big|$.
The estimates (\ref{strongerest}) prove that $\big\Vert T_{[a^{(j)}]}f\big\Vert_{B_{p}^{s_2,q}}\lesssim \Vert f\Vert_{B_{p}^{s_1,q}}$ for each $j=1,2$, and the boundedness of $T_{[a^{(3)}]}$ is an immediate consequence of (\ref{previous}) and (\ref{inftyneed}).


\section{\textbf{Proof of Theorem \ref{sharptheorem} and \ref{sharptheorem2}}}

We construct multiplier operators and  since $s_1$ and $s_2$ do not affect the boundedness of multipliers, we assume $s_1=s_2=0$. 
As described earlier, we need to prove Theorem \ref{sharptheorem} (3), (4), and Theorem \ref{sharptheorem2} (2), assuming $m=-d\big| 1/2-1/p\big|$.

Recall the Khintchine's inequality;
For a (countable) index set $\mathcal{I}$ , let $\{r_n\}_{n\in \mathcal{I}}$ be the Rademacher functions defined on $[0,1]$ and $\{c_n\}_{n\in\mathcal{I}}$ be a sequence of complex numbers. 
Then for $0<p<\infty$,
\begin{equation*}
\Big(  \int_0^1{\big| \sum_{n\in\mathcal{I}}{c_nr_n(v)}\big|^p}dv  \Big)^{1/p} \approx \Big(  \sum_{n\in\mathcal{I}}{|c_n|^2}  \Big)^{1/2}.
\end{equation*}

For each $k\in\mathbb{N}$ let $\zeta_k:=10k$ and  $\mathcal{N}_k:=\{n\in\mathbb{Z}^d : 2^{\zeta_k+2}\leq |n|<2^{\zeta_k+3}\}$.
 Now let $\{r_n\}_{n\in\mathcal{N}_k}$ be the Rademacher functions and define \begin{equation*}
M^v(\xi):=\sum_{k=10}^{\infty}{2^{\zeta_km}\sum_{n\in\mathcal{N}_k}{r_n(v)\widehat{\phi}(\xi-n)}}.
\end{equation*} 
Then 
$ \big\{ 2^{\zeta_km}\sum_{n\in\mathcal{N}_{k}}{r_n(v)\widehat{\phi}(\xi-n)}\big\}_{k=0}^{\infty}$ is a collection of functions having pairwise disjoint compact support $\{\xi\in\mathbb{R}^d:|\xi|\approx 2^{\zeta_k}\}$.
Moreover, the differentiation does not change the decay at all. 
Since $|r_n(v)|=1$ for all $n\in\mathbb{Z}^d$ and $v\in [0,1]$,  $M^v \in \mathcal{S}_{0,0}^{m}$ uniformly in $v\in[0,1]$.

Let $\mathcal{G}(x)=\sum_{j=1}^{4}{\phi_j(x)}$. Then its Fourier transform equals $1$ on $\{\xi\in\mathbb{R}^d:2\leq |\xi|\leq 2^4\}$, vanishes outside $\{\xi\in\mathbb{R}^d:1<|\xi|<2^5\}$.
This implies that 
\begin{equation}\label{supporttrick}
\widehat{\mathcal{G}}(\xi/2^{\zeta_k})\widehat{\phi}(\xi-n)=\widehat{\phi}(\xi-n), \quad \text{for}~ n\in\mathcal{N}_k.
\end{equation}

\subsection{Proof of Theorem \ref{sharptheorem} (2)}
Assume $0<t<p\leq 2$ and $m=-d(1/p-1/2)$.
We shall use a randomization technique in \cite{Ch_Se}.

For each $k\in\mathbb{N}$ let $\mathcal{Q}(k)$ be the subset of $\mathcal{D}_{\zeta_k}$ contained in $[0,1]^d$. Let $\Omega$ be a probability space with probability measure $\mu$ and let $\{\theta_Q\}$ be a family of independent random variables, each of which takes the value $1$ with probability $A_k$ and the value $0$ with probability $1-A_k$ when $Q\in\mathcal{Q}(k)$. Here $A_k$ is, of course, a constant between $0$ and $1$.

For each $w\in\Omega$ define \begin{equation*}
f^{L,w}(x):=\sum_{k=10}^{L}{B_k\sum_{Q\in\mathcal{Q}(k)}{\theta_Q(w)\mathcal{G}(2^{\zeta_k}(x-c_Q))}}
\end{equation*} where $c_Q$ is the center of $Q$ and $\{B_k\}$ is a sequence of positive numbers increasing at least in a geometric progression. 
Then as shown in \cite[Proof of Theorem1.4]{Ch_Se}, 
\begin{equation}\label{ubound}
\Big(  \int_{\Omega}{\Big\Vert   f^{L,w} \Big\Vert_{F_p^{0,q}}^p}d\mu(w) \Big)^{1/p} \lesssim \Big( \sum_{k=10}^{L}{B_k^p A_k}   \Big)^{1/p}.
\end{equation}
By putting $A_k=2^{-\zeta_kd}$ and $B_k=2^{\zeta_kd/p}$, \begin{equation*}
(\ref{ubound})\lesssim L^{1/p}. 
\end{equation*}

Moreover, due to (\ref{supporttrick}) one can write
\begin{equation*}
M^v(D)f^{L,w}(x)=\sum_{k=10}^{L}{B_k2^{\zeta_k(m-d)}\sum_{Q\in\mathcal{Q}(k)}{\theta_Q(w)\phi(x-c_Q)\sum_{n\in\mathcal{N}_k}{r_n(v)e^{2\pi i<x-c_Q,n>}}}}.
\end{equation*} 
Our claim is that
\begin{equation}\label{show}
\Big(\int_0^1{\int_{\Omega}{\big\Vert  M^v(D)f^{L,w} \big\Vert_{F_p^{0,t}}^p}d\mu(w)}dv\Big)^{1/p} \gtrsim L^{1/t}
\end{equation} and this implies that there exists $v_0\in [0,1]$ such that 
\begin{equation*}
\sup\{\Vert   M^{v_0}(D)f \Vert_{F_p^{0,t}} : \Vert f\Vert_{F_p^{0,q}}\leq 1, f\in\mathcal{E}(r)\}\gtrsim (\log{r})^{1/t-1/p},
\end{equation*}  from which the desired result follows.

Let's prove (\ref{show}). Note that 
\begin{align*}
&\big\Vert  M^v(D)f^{L,w} \big\Vert_{F_p^{0,t}}\\
&\approx \Big\Vert  \Big( \sum_{k=10}^{L}{B_{k}^t 2^{\zeta_kt(m-d)} \Big|   \sum_{Q\in\mathcal{Q}(k)}{\theta_Q(w)\phi(\cdot-c_Q) \sum_{n\in\mathcal{N}_k}{r_n(v)e^{2\pi i<\cdot-c_Q,n>}}      }   \Big|^t    }   \Big)^{1/t}  \Big\Vert_{L^p}.
\end{align*}
Let $\Omega(k,Q)$ be the event that $\theta_Q(w)=1$ but $\theta_{Q'}=0$ for all $Q'\not= Q$ in $\mathcal{Q}(k)$. The probability of this event satisfies
\begin{equation}\label{prob}
\mu(\Omega(k,Q))\geq A_k(1-A_k)^{card(\mathcal{Q}(k))-1}\gtrsim A_k
\end{equation} with $card(\mathcal{Q}(k))=2^{\zeta_kd}$ and our choice $A_k=2^{-\zeta_kd}$.
Here, the constant in the inequality does not depend on $k$ as $\lim_{M\to \infty}{\big( 1-\frac{1}{M}\big)^M}=e^{-1}$.

Then one has
\begin{align*}
&\Big(\int_0^1{\int_{\Omega}{\big\Vert  M^v(D)f^{L,w} \big\Vert_{F_p^{0,t}}^p}d\mu}dv\Big)^{1/p}\\
&\gtrsim \Big(  \sum_{k=10}^{L}{B_k^t2^{\zeta_kt(m-d)}\int_{[0,1]^d}{\int_0^1{\int_{\Omega}{\Big|  \sum_{Q\in\mathcal{Q}(k)}{\theta_Q(w)\phi(x-c_Q)\sum_{n\in\mathcal{N}_k}{r_n(v)e^{2\pi i<x-c_Q,n>}}}  \Big|^t}d\mu}dv}dx}   \Big)^{1/t}\\
&\gtrsim \Big(  \sum_{k=10}^{L}{B_k^t2^{\zeta_kt(m-d)}\sum_{Q\in\mathcal{Q}(k)}{  \mu(\Omega(k,Q))\int_{[0,1]^d}{|\phi(x-c_Q)|^t\int_0^1{\Big|  \sum_{n\in\mathcal{N}_k}{r_n(v)e^{2\pi i<x-c_Q,n>}}  \Big|^t} dv}dx   }}   \Big)^{1/t}\\
&\gtrsim \Big(  \sum_{k=10}^{L}{B_k^tA_k2^{\zeta_kt(m-d)}\sum_{Q\in\mathcal{Q}(k)}{  \int_{[0,1]^d}{|\phi(x-c_Q)|^t\int_0^1{\Big|  \sum_{n\in\mathcal{N}_k}{r_n(v)e^{2\pi i<x-c_Q,n>}}  \Big|^t} dv}dx   }}   \Big)^{1/t}
\end{align*} where the first inequality follows from H\"older's inequality with $p/t>1$, the second one from the estimate $\displaystyle\int_{\Omega}{|\cdots|}d\mu\geq \sum_{Q\in\mathcal{Q}(k)}{\int_{\Omega(k,Q)}{|\cdots|}d\mu}$, and the last one from (\ref{prob}). 

Finally, we apply the Khintchine's inequality with $card(\mathcal{N}_k)\sim2^{\zeta_kd}$ to get the lower bound
\begin{align*}
&\Big(  \sum_{k=10}^{L}{B_k^tA_k2^{\zeta_kt(m-d/2)}\sum_{Q\in\mathcal{Q}(k)}{\int_{[0,1]^d}{|\phi(x-c_Q)|^t}dx}}  \Big)^{1/t}\\
&\gtrsim \Big( \sum_{k=10}^{L}{B_k^tA_k2^{\zeta_kt(m-d/2)}2^{\zeta_kd}}\Big)^{1/t}\gtrsim L^{1/t}
\end{align*} because $|\phi(x-c_Q)|\gtrsim 1$ uniformly in $x, c_Q\in [0,1]^d$, and $m=-d(1/p-1/2)$, $A_k=2^{-\zeta_kd}$, and $B_k=2^{\zeta_kd/p}$.

\subsection{Proof of Theorem \ref{sharptheorem} (3)}

Suppose $2\leq p<q\leq \infty$ and $m=-d(1/2-1/p)$.
In this case, by using the same duality technique in \cite[Section 6.1.3]{Park}  one can obtains
\begin{equation*}
\sup\{\Vert   \big(M^{v_0}(D)\big)^*f \Vert_{F_p^{0,\infty}} : \Vert f\Vert_{F_p^{0,q}}\leq 1, f\in\mathcal{E}(R^A)\}\gtrsim_{\epsilon} (\log{R})^{\epsilon}
\end{equation*}
for any $0<\epsilon<1/p-1/q$ and some $A>0$, and this proves Theorem \ref{sharptheorem} (3).

\subsection{Proof of Theorem \ref{sharptheorem2} (2)}

We first assume $0<p\leq 2$ and $0<t<q\leq \infty$.
Define \begin{equation*}
g^{L}(x):=\sum_{k=10}^{L}{C_k 2^{\zeta_kd} \mathcal{G}(2^{\zeta_k}x)}.
\end{equation*}
Then
\begin{equation}\label{upperest}
\Vert g^L\Vert_{B_p^{0,q}} \lesssim \big\Vert \big\{ C_k\Vert 2^{\zeta_kd}\mathcal{G}(2^{\zeta_k}\cdot)\Vert_{L^p}\big\}_{k=10}^{L}\big\Vert_{l^q}\approx \big\Vert \big\{C_k2^{\zeta_kd(1-1/p)} \big\}_{k=10}^{L}\big\Vert_{l^q}.
\end{equation}

Moreover, by using (\ref{supporttrick})
\begin{equation*}
M^{v}(D)g^L(x)=\phi(x)\sum_{k=10}^{L}{C_k2^{\zeta_km}\sum_{n\in\mathcal{N}_k}{r_n(v)e^{2\pi i\langle x,n\rangle}}}
\end{equation*}
and thus
\begin{equation*}
\Big( \int_{0}^{1}{\big\Vert M^{v}(D)g^L\big\Vert_{B_p^{0,t}}^t}dv\Big)^{1/t} \approx \Big(\sum_{k=10}^{L}{C_k^t2^{\zeta_kmt}\int_0^1{\Big\Vert \phi\cdot\sum_{n\in\mathcal{N}_k}{r_n(v)e^{2\pi i\langle \cdot,n\rangle}}\Big\Vert_{L^p}^t}dv} \Big)^{1/t}.
\end{equation*}
Observe that
\begin{align*}
&\int_0^1{\Big\Vert \phi\cdot\sum_{n\in\mathcal{N}_k}{r_n(v)e^{2\pi i\langle \cdot,n\rangle}}\Big\Vert_{L^p}^t}dv\\
&\gtrsim \Big(\int_{\mathbb{R}^d}{|\phi(x)|^p\Big(\int_0^1{\Big| \sum_{n\in\mathcal{N}_k}{r_n(v)e^{2\pi i\langle x,n\rangle}}\Big|^{\min{(t,p)}}}dv\Big)^{p/\min{(t,p)}}}dx \Big)^{t/p}
\end{align*} (due to H\"older's inequality when $t>p$ or Minkowski inequality when $t<p$) and this expression is comparable to $2^{\zeta_kdt/2}$ by applying Khintchine's inequality. Therefore
one has
\begin{equation}\label{lowerest}
\Big( \int_{0}^{1}{\big\Vert M^{v}(D)g^L\big\Vert_{B_p^{0,t}}^t}dv\Big)^{1/t}\gtrsim \big\Vert \big\{C_k2^{\zeta_kd(1-1/p)} \big\}_{k=10}^{L}\big\Vert_{l^t}
\end{equation}
Finally, we are done by choosing $C_k=k^{-1/t}2^{-\zeta_kd(1-1/p)}$ so that $(\ref{upperest})\lesssim 1$ and $(\ref{lowerest})\approx \ln{L}$.

When $2<p<\infty$, the result follows from the duality arguments in \cite[6.2.2]{Park}.

\section*{Acknowledgement}

{The author would like to thank Andreas Seeger for the guidance and helpful discussions. The author also thanks the referee for carefully reading the paper and making numerous useful remarks.
The author is supported in part by NRF grant 2019R1F1A1044075 and was supported in part by NSF grant DMS 1500162}.

\end{document}